\documentclass[letterpaper]{article}
\usepackage{soul}
\usepackage{graphicx,xcolor}
\usepackage{amsmath,amsthm}
\usepackage{amssymb}
\usepackage{float}
\usepackage[hidelinks]{hyperref}
\usepackage{enumerate}
\usepackage{mathrsfs}

\usepackage[english]{babel}

\theoremstyle{plain}

\newtheorem*{theorem*}{Theorem}
\newtheorem{theorem}{Theorem}[section]
\newtheorem{claim}{Claim}[section]
\newtheorem{definition}{Definition}[section]
\newtheorem{lemma}[theorem]{Lemma}
\newtheorem{proposition}[theorem]{Proposition}
\newtheorem{corollary}[theorem]{Corollary}

\newtheorem*{question*}{Question}
\newtheorem*{questions*}{Questions}
\newtheorem*{remark*}{Remark}

\def\NN{\mathbb{N}}
\def\ZZ{\mathbb{Z}}

\def\ag{\mathcal{A}}

\def\CC{\mathcal{C}}

\def\FF{\mathcal{F}}
\def\dens{\operatorname{dens}}

\newcommand{\define}[1]{\emph{#1}}

\newcommand{\cor}[2][]{#2}

\usepackage{tikz}
\usetikzlibrary{patterns,positioning,arrows,decorations.markings,calc,decorations.pathmorphing,decorations.pathreplacing}
\usetikzlibrary{lindenmayersystems}
\usetikzlibrary{arrows}
\usepackage{tikz-3dplot}

\pgfdeclarelindenmayersystem{cayleyPSL2Z}{
	\rule{A -> [FB+FF[+A]----FF[+A]----FF]++++++}
	\rule{B -> FBB}}

\definecolor{vert}{RGB}{0,178,102}

\title{Realization of aperiodic subshifts and uniform densities in groups}
\date{}
\author{Nathalie Aubrun\thanks{LIP, ENS de Lyon -- CNRS -- INRIA -- UCBL -- Universit\'e de Lyon}~, Sebasti\'an Barbieri\footnotemark[1]\\ and St\'ephan Thomass\'e\footnotemark[1]~\thanks{Institut Universitaire de France} \\
	\texttt{nathalie.aubrun@ens-lyon.fr}\\ \texttt{sebastian.barbieri@ens-lyon.fr}\\ \texttt{stephan.thomasse@ens-lyon.fr}}

\begin{document}
	
	\maketitle 

	\begin{abstract}
		A theorem of Gao, Jackson and Seward, originally conjectured to be false by Glasner and Uspenskij, asserts that every countable group admits a 2-coloring. A direct consequence of this result is that every countable group has a strongly aperiodic subshift on the alphabet $\{0,1\}$. In this article, we use Lov\'asz local lemma to first give a new simple proof of said theorem, and second to prove the existence of a $G$-effectively closed strongly aperiodic subshift for any finitely generated group $G$. We also study the problem of constructing subshifts which generalize a property of Sturmian sequences to finitely generated groups. More precisely, a subshift over the alphabet $\{0,1\}$ has uniform density $\alpha \in [0,1]$ if for every configuration the density of 1's in any increasing sequence of balls converges to $\alpha$. We show a slightly more general result which implies that these subshifts always exist in the case of groups of subexponential growth.
	\end{abstract}	
	
	
	\section*{Introduction}
	\label{section.introduction}
	
	Symbolic dynamics is concerned with the study of subshifts on groups. Subshifts are sets of colorings of a group $G$ by some finite alphabet $\ag$ that respect local constraints given by forbidden patterns, or equivalently, subsets of $\ag^G$ that are both closed for the product topology and shift-invariant. They can be used to model dynamical systems~\cite{MorseHedlund1938,pollicott1998dynamical}, but can also be seen as computational models~\cite{Hochman2009b}. Subshifts of finite type (SFT) -- those which can be defined by forbidding a finite set of patterns -- constitute an interesting class of subshifts since they are defined by local conditions, and can model real-world phenomena. Classical symbolic dynamics has focused on the one-dimensional case $G=\ZZ$~\cite{MorseHedlund1938,lind1995introduction} and more recently $G=\ZZ^d$ with $d\geq2$~\cite{Hochman2009b,PavlovSchraudner2009}. Subshifts on free groups have also been studied~\cite{Piantadosi2006,Piantadosi2008}. Very recent results tackle 
computational aspects of subshifts on finitely generated groups~\cite{ABS2014, BallierStein2013,Cohen,CarrollPenland,Jeandel2015}.
	
	\bigskip
	
	In this article we consider two general aspects of realizability which concern subshifts in groups: The first one asks if it is possible to construct a strongly aperiodic subshift, that is, one such that the stabilizer of every element is the trivial subgroup. The second aspect is inspired by Sturmian words, by the fact that the factors of length $n$ carry a density of 1's which converges to the slope of the irrational rotation which generates the word. Here, given a finitely generated group and its word metric given by a set of generators, we ask if \cor[it's]{it is} possible to construct a subshift $X \subset \{0,1\}^G$ such that for every configuration $x \in X$ and every sequence $(g_n)_{n \in \NN}$ of group elements, the density of 1's over the balls $B(g_n,n)$ always converges to a fixed density $\alpha \in [0,1].$ We call this property \emph{uniform density}.
	
	The existence of a countable group which does not admit a non-empty strongly aperiodic subshift over the alphabet $\{0,1\}$ was asked in~\cite{glasner2009} and subsequently answered negatively in~\cite{gao2009}. Nevertheless, their proof is very technical. In this article we combine the asymmetric version of Lov\'asz local lemma~\cite{AlonSpencer2008} and compactness of the set of configurations to get a nice tool to prove non-emptiness of subshifts defined by forbidden patterns. This technique, which in some sense is the analogue of the probabilistic method in graph theory, provides very short proofs of the existence of configurations in subshifts. We use it to prove again in a succinct way the existence of a strongly aperiodic subshift on any countable group. We also extend the previous result by showing that in finitely generated groups it is also possible to construct non-empty strongly aperiodic subshifts which satisfy the condition of being $G$-effectively closed, that is, that they can be defined as 
the complement of a recursively enumerable union of cylinders by a Turing machine which has access to the word problem of $G$. That is, we show:

	{
		\renewcommand{\thetheorem}{\ref{theorem.free_subflow}}
		\begin{theorem}
		Every countable group $G$ has a non-empty, strongly aperiodic subshift on the alphabet $\{0,1\}$.
		\end{theorem}
		\addtocounter{theorem}{-1}
	}

	{
		\renewcommand{\thetheorem}{\ref{theorem.strongly_aperiodic1}}
		\begin{theorem}
			Every finitely generated group $G$ has a non-empty $G$-effectively closed strongly aperiodic subshift.
		\end{theorem}
		\addtocounter{theorem}{-1}
	}

	A bi-infinite sequence of $0$'s and $1$'s is balanced if for every $n\in\NN$, every factor of size $n$ can have only two possible quantities of symbols $1$. Most famous examples of balanced sequences are Sturmian sequences, which are both balanced and aperiodic, or equivalently, codifications of irrational rotations in the circle~\cite{Lothaire,fogg2002substitutions}. Some attempts have been made to generalize Sturmian sequences to dimension~2~\cite{BertheVuillon2000,Fernique2006}. In general groups a configuration such that the amount of $1$'s over any finite connected support of size $n$ has at most two values is not possible as the group's geometry is too wild. Instead, we propose this uniform density property of a subshift \cor{mentioned above}{} as a generalization of the aforementioned property of Sturmian subshifts. In this direction, we are able to prove a more technical result which implies the existence of non-empty subshifts with uniform density in every group of subexponential growth. Formally, 
we show the following result\cor[ which makes sense in amenable groups]{}.
	{
		\renewcommand{\thetheorem}{\ref{theorem.densities}}
		\begin{theorem}
			Let $G$ be an infinite and finitely generated \cor{amenable} group and $\alpha \in [0,1]$. \cor[There is a non-empty subshift $X_{\alpha} \subset \{0,1\}^G$ such that for any $x \in X_{\alpha}$ and F\o lner sequence $(F_n)_{n \in \NN}$ then $\lim_{n \to \infty} \dens(1,F_n,x) = \alpha$]{Then there exists a non-empty subshift $X_{\alpha} \subset \{0,1\}^G$ with $\lim_{n \to \infty} \dens(1,F_n,x) = \alpha$ for any $x \in X_{\alpha}$ and any F\o lner sequence $(F_n)_{n \in \NN}$}.
		\end{theorem}
		\addtocounter{theorem}{-1}
	}
	
	That is, by replacing the sequence of intervals of $\ZZ$ for a general F\o lner sequence we produce a subshift where the densities of 1's converge to a fixed value $\alpha$. We show that the subshifts given by our construction are weakly aperiodic if $\alpha \notin \mathbb{Q}$. Our example loses some of the properties of those constructions, like balancedness, but instead uses a two-symbol alphabet and is available on any infinite, finitely generated and amenable group. 

	\section{Preliminaries}
	\label{section.preliminaries}
	
	Throughout this article the groups $G$ considered will \cor[be either countable or finitely generated]{always be countable}; we denote their identity element $1_G$. When $G$ is finitely generated we associate a finite set $S\subset G$ of generators and the undirected right Cayley graph $\Gamma(G,S) = (G, \{ \{g,gs\} \mid g \in G, s\in S\})$ so that $(G,d)$ is a metric space where $d$ is the distance induced on $G$ by $\Gamma(G,S)$. \cor{If two words $w_1,w_2$ in $S^*$ represent the same element in $G$, we write $w_1 =_G w_2$.} We shall denote by $B(g,n) = \{h \in G \mid d(g,h) \leq n\}$ the ball of size $n$ centered in $g \in G$. In general we denote $B_{\Gamma}(v,n)$ the ball of size $n$ centered in $v$ of an arbitrary graph $\Gamma$. \cor{For $g\in G$, we denote by $|g|$ the length of a shortest path from $1_G$ to $g$ in $\Gamma(G,S)$, that is to say $|g|=d(1_G,g)$.} We also denote by $\texttt{WP}(G) := \{ w \in (S\cup S^{-1})^{*} \mid w =_G 1_G \}$ the set of words which can be written using elements 
from $S$ and their inverses which are equal to $1_G$ in the group $G$. If $\texttt{WP}(G)$ is a decidable language we say that $G$ has decidable word problem. For more references see~\cite{MagnusKarassSolitar2004}.

	We now give some basic definitions of symbolic dynamics. For a more complete introduction the reader may refer to~\cite{lind1995introduction,ceccherini-SilbersteinC09}. Let $\ag$ be a finite alphabet and $G$ a countable group. The set $\ag^G = \{ x: G \to \ag\}$ equipped with the left group action $\sigma: G \times \ag^G \to \ag^G$ defined by $(\sigma_g(x))_h = x_{g^{-1}h}$ is the \textit{$G$-fullshift}. The elements $a \in \ag$ and $x \in \ag^G$ are called \define{symbols} and \define{configurations} respectively. By taking the discrete topology on $\ag$ we obtain that the set of configurations $\ag^G$ is compact and metrizable. In the case of a countable group, given an enumeration $1_G = g_0,g_1,\dots$ of $G$, the topology is generated by the metric $d(x,y) = 2^{-\inf (\{ n \in \NN \mid\ x_{g_n} \neq y_{g_n}  \})}$. If $E$ is a subset of $\ag^G$, we denote by~$\overline{E}$ its topological closure. In the case of a finitely generated group another possibility which is more practical is $\
\displaystyle{d(x,y) = 2^{-\inf\{|g|\; \mid\; g \in G:\; x_g \neq y_g\}}}$. This topology is generated by a clopen basis given by the \define{cylinders} $[a]_g = \{x \in \ag^G | x_g = a\in \ag\}$. A \emph{support} is a finite subset $F \subset G$. Given a support $F$, a \emph{pattern with support $F$} is an element $p$ of $\ag^F$, i.e. a finite configuration and we write $supp(p) = F$. We also denote the cylinder generated by $p$ centered in $g$ as $[p]_g = \bigcap_{h \in F}[p_h]_{gh}$
	One says that a pattern $p\in \ag^{F}$ \emph{appears} in a configuration $x \in \ag^G$ if there exists $g \in G$ such that for any $h \in F$, $x_{gh} = p_h$, said otherwise, if there exists $g$ such that $x \in [p]_g$. In this case we write $p \sqsubset x$. We denote the set of finite patterns over $G$ as $\ag_G^* := \bigcup_{F \subset G, |F| < \infty}{\ag^F}$.
	\begin{definition}
		A subset $X$ of $\ag^G$ is a \define{$G$-subshift} if it is $\sigma$-invariant -- $\sigma(X)\subset X$ -- and closed for the cylinder topology. Equivalently, $X$ is a $G$-subshift if and only if there exists a set of forbidden patterns $\FF \subset \ag_G^*$ that defines it.
		$$X=X_\FF := \left\{\ x\in \ag^G\mid \forall p \in \FF, p \not\sqsubset x \right\} = \bigcap_{p \in \FF, g \in G}{\ag^G \setminus [p]_g}.$$
	\end{definition}
	
	That is, a $G$-subshift is a shift-invariant subset of $\ag^G$ which can be written as the complement of a union of cylinders.
	If \cor[the context is clear enough]{it is clear from the context}, we will drop the $G$ and simply refer to a subshift. A subshift $X\subseteq \ag^G$ is \define{of finite type} -- $G$-SFT for short -- if there exists a finite set of forbidden patterns $\FF$ such that $X=X_\FF$. 
	
	Consider a group which is generated by a finite set $S$. A \define{pattern coding} $c$ is a finite set of tuples $c=(w_i,a_i)_{1 \leq i \leq n}$ where $w_i \in (S \cup S^{-1})^{*}$ and $a_i \in \ag$. We say that a pattern coding is \define{consistent} if for every pair of tuples such that $w_i =_G w_j$ ($w_i$ and $w_j$ represent the same element under $G$) then $a_i = a_j$. We say a consistent pattern coding $c$ \define{codifies} a pattern $P$ if every $w_i$ represents an element of $supp(P)$ and for every $g \in supp(P)$ there exists a tuple $(w_i,a_i) \in c$ such that $g =_G w_i$ and $P_g = a_i$.
	
	\begin{definition}
		For a finitely generated group $G$ we say a subshift $X\subseteq \ag^G$ is \define{$G$-effectively closed} if there exists a Turing machine with oracle $\texttt{WP}(G)$ which recognizes a set of pattern codings such that the consistent ones codify a set of patterns $\FF$ such that $X=X_\FF$. If the same property is valid without the oracle we say $X$ is \define{effectively closed}.
	\end{definition}
	
	Being $G$-effectively closed is a generalization of the same concept for $\ZZ$-subshifts where the set of forbidden patterns is a recursively enumerable set of words. 
	
	
	Let $x\in \ag^G$ be a configuration. The \define{orbit} of $x$ is the set of configurations $orb_\sigma(x)=\left\{\sigma_g(x)\mid g\in G \right\}$, and the \define{stabilizer} of $x$ is the set of group elements $stab_\sigma(x)=\left\{\ g\in G\mid \sigma_g(x)=x\right\}$. \cor[In the context of subshifts, the stabilizer is a normal subgroup of $G$.]{}
	
	\begin{definition}
		A $G$-subshift $X\subseteq \ag^G$ is \define{weakly aperiodic} if for every configuration $x\in X$, $|orb_\sigma(x)|=\infty$. A $G$-subshift $X\subseteq \ag^G$ is \define{strongly aperiodic} if for every configuration $x\in X$, $stab_\sigma(x)=\left\{ 1_G\right\}$.
	\end{definition}
	
	For infinite groups the weak concept of aperiodicity is relevant and implied by strong aperiodicity.
	
	
	\section{Non-empty strongly aperiodic subshifts}
	\label{section.strongly_aperiodic_subshifts}
	
	In this section we construct non-empty strongly aperiodic subshifts on any countable group. \cor{As mentionned in the introduction,} \cor[T]{t}he question of the existence of an infinite countable group $G$ that does not admit a non-empty strongly aperiodic subshift over the alphabet $\{0,1\}$ was asked in~\cite{glasner2009} and subsequently answered in the negative in~\cite{gao2009}. \cor[The proof presented in the latter article is a technical construction.]{} In this section we present a short proof based on Lov\'asz local lemma~\cite{AlonSpencer2008}. \cor[We]{For finitely generated groups, we also} give a second proof -- inspired by the use of local lemma in~\cite{alonetal_nonrepetitivecoloringofgraphs} -- which is quite easy to visualize and gives a $G$-effectively closed subshift, but this proof uses a large alphabet and only works in finitely generated groups.
	
	We begin by introducing the asymmetric version of the local lemma. We then extract a corollary to show how it can be used in order to produce configurations in subshifts by using the compactness of the set of configurations and then we proceed to the construction of the strongly aperiodic subshifts.
	
	\subsection{Lov\'asz local lemma}
	\label{subsection.LLL}

	\begin{lemma}\label{lovaszlocallemma}[Asymmetric Lov\'asz local lemma]
		Let $\mathscr{A} := \{A_1,A_2,\dots,A_n\}$ be a finite collection of measurable sets in a probability space $(X,\mu, \mathcal{B})$. For $A \in \mathscr{A}$, let $\Gamma(A)$ be the \cor{smallest} subset of $\mathscr{A}$ such that $A$ is independent of the collection $\mathscr{A} \setminus (\{A\} \cup \Gamma(A)).$ Suppose there exists a function $x: \mathscr{A} \to (0,1)$ such that: $$\forall A \in \mathscr{A}: \mu(A) \leq x(A)\prod_{B \in \Gamma(A)}(1-x(B))$$ then the probability of avoiding all events in $\mathscr{A}$ is positive\cor[, in particular:]{. Specifically} $$ \mu\left(  X \setminus \bigcup_{i = 1}^{n} {A_i}\right) \geq \prod_{A \in \mathscr{A}}(1-x(A)) > 0.$$
	\end{lemma}
	
	The sets $A_1,A_2,\dots,A_n$ can be seen as bad events that we want to avoid. In the context of the present article where $\ag$ is a finite alphabet and $G$ a countable group, we will choose the probability space to be $X=\ag^G$ with $\mu$ any Bernoulli probability measure on $\ag^G$. Suppose $X$ is a subshift defined by a set of forbidden patterns $\FF= \bigcup_{n \geq 1}\FF_n$ where $\FF_n \subset \ag^{S_n}$ is a finite set of patterns with \cor[a]{} finite support $S_n$. We will consider the bad events $A_{n,g}= \bigcup_{p \in \FF_n}[p]_g = \left\{ x\in\ag^G: x|_{gS_n} \in \FF_n \right\}$, that is to say one of the forbidden patterns $p \in \FF_n$ appears in position $g$. Subshifts might be defined by an infinite amount of forbidden patterns while the lemma only holds for a finite collection of bad events. Nevertheless the compactness of $\ag^G$ allows us to use the lemma anyway, as explained in what follows.
	
	\begin{lemma}
		\label{corollary.LLL_nonemptiness} 
		Let $G$ a countable group and $X\subset \ag^G$ a subshift defined by the set of forbidden patterns $\FF= \bigcup_{n \geq 1}\FF_n$, where $\FF_n \subset \ag^{S_n}$. \cor{Let $\mu$ be a Bernoulli probability measure on $\ag^G$.} Suppose that there exists a function $x: \NN\times G \to (0,1)$ such that:
		\begin{equation}\label{eq.condition_LLL}\tag{$*$}
		\forall n\in\NN, g\in G,\; \mu(A_{n,g})\leq x(n,g)\prod_{\substack{gS_n\cap hS_k\neq\emptyset\\ (k,h) \neq (n,g)}}(1-x(k,h)),
		\end{equation}
		where $A_{n,g}=\left\{ x\in\ag^G: x|_{gS_n} \in \FF_n \right\}$ \cor[and $\mu$ is any Bernoulli probability measure on $\ag^G$]{}. Then the subshift $X$ is non-empty. 
	\end{lemma}
	
	\begin{proof}
		Consider an enumeration $\{g_k\}_{k\in\NN}$ of $G$. For every $n\in\NN$, we apply Lemma~\ref{lovaszlocallemma} to construct a configuration $x_n\in\ag^G$ that satisfies the following property: for every forbidden pattern $p \in \FF_k$ such that $k \leq n$ and every $g \in (g_k)_{k\leq n}$ such that $gS_k\subseteq \{g_k\}_{k\leq n}$, we have $(x_n) \notin [p]_g$ -- in other terms, the configuration $x_n$ avoids all the forbidden patterns in $\bigcup_{k \leq n}\FF_k$ on the finite set $\{g_1,\dots,g_n\}\subset G$. Indeed, in order to show the existence of $x_n$ we only need that for every $k \leq n$ and $g \in G$ such that $gS_k\subseteq \{g_k\}_{k\leq n}$,
		$$ \mu(A_{k,g})\leq x(k,g)\prod_{\substack{gS_k\cap hS_{k'}\neq\emptyset\\ hS_{k'}\subseteq \{g_0,\dots,g_n\} \\ (k',h) \neq (k,g), k \leq n}   }(1-x(k',h))$$
		which is a relaxation of condition~\eqref{eq.condition_LLL} by the fact that $0\leq x(k',h)\leq 1$. The local lemma holds since the set $\{g_0,\dots,g_n\}$ is finite and we only consider a finite number of forbidden patterns, consequently we only consider a finite number of bad events $A_{k,g}$.
		
		By compactness, we can extract from this sequence of configurations $(x_n)_{n\in\NN}$ a subsequence $(x_{\phi(n)})_{n\in\NN}$ converging to some $x \in \ag^G$. Then $x$ does not contain any forbidden pattern $p \in \FF$. Suppose it were the case, that is to say, there exists some $g\in G$ and $p \in \FF_m$ such that $x \in [p]_g$. Since there exists also $n,l$ such that $g = g_l$ and $gS_m\subset\{g_0,\dots,g_n\}$, by definition of the metric there exists some $N \geq \max\{m,n,l\}$ sufficiently big \cor[which with this property that]{such that $\phi(N)$} appears in the subsequence $(\phi(n))_{n\in\NN}$. Then $x_N$ contains the forbidden pattern $p$ somewhere in $(x_N)|_{\{g_1,\dots,g_N\}}$. This contradicts the construction of the sequence $(x_n)_{n\in\NN}$, thus $x \in X_{\FF}$.
	\end{proof}
	
	\subsection{A non-empty strongly aperiodic subshift over $\{0,1\}$ in any countable group.}
	\label{subsection.simpler_proof}
	
	Consider a configuration $x \in \{0,1\}^G$. We say that $x$ has \define{the distinct neighborhood property} -- in \cite{gao2009} they call $x$ a $2$-coloring -- if for every $h \in G \setminus \{1_G\}$ there exists a finite subset $T \subset G$ such that:
	
	$$\forall g \in G : x|_{ghT} \neq x|_{gT}.$$
	
	\begin{proposition}
		If a configuration $x \in \{0,1\}^G$ has the distinct neighborhood property, then the $G$-subshift $X := \overline{orb_\sigma(x)}$ is strongly aperiodic.
	\end{proposition}
	
	\begin{proof}
		
		Let $y \in X$. By definition there exists a sequence $(g_i)_{i \in \NN}$ such that $\sigma_{g_i}(x)$ converges to $y$ in the product topology. Suppose there is $h \neq 1_G$ such that $\sigma_h(y) = y$, then by continuity of the shift action under the product topology we have that $\sigma_{hg_i}(x) \to \sigma_{h}(y) = y$. Since $x$ has the distinct neighborhood property, there exists a finite subset $T$ of $G$ -- associated to $h^{-1}$ -- such that $\forall g \in G : x|_{gh^{-1}T} \neq x|_{gT}$. By definition of convergence in the metric, there exists $n \in \NN$ such that $T \subset \{g_0, g_1, \dots, g_n\}$ and $m \in \NN$ satisfying:
		
		$$\sigma_{hg_m}(x)|_{\{g_0, g_1, \dots, g_n\}} = y|_{\{g_0, g_1, \dots, g_n\}} = \sigma_{g_m}(x)|_{\{g_0, g_1, \dots, g_n\}}$$
		
		Therefore $\sigma_{hg_m}(x)|_T = \sigma_{g_m}(x)|_T$ which implies that $x|_{g_m^{-1}h^{-1}T} = x|_{g_m^{-1}T}$, a contradiction.	\end{proof}
	
	\begin{theorem}
		\label{theorem.free_subflow}
		Every countable group $G$ has a non-empty, strongly aperiodic subshift on the alphabet $\{0,1\}$.
	\end{theorem}
	
	\begin{proof}
	
		The case where $G$ is finite is trivial, as the $G$-SFT given by $$X := \{x \in \{0,1\}^G \mid |x^{-1}(1)| = 1 \}$$ is strongly aperiodic. Indeed, let $x \in X$ and $g \in G$ be the only element such that $x_g = 1$. Let $h \in stab_{\sigma}(x)$ then $\sigma_{h}(x) = x$ which implies that $x_{h^{-1}g}= x_g = 1$ and thus $h = 1_G$. For the rest of the proof we suppose that $G$ is infinite.
		
		Let $(s_i)_{i\in\NN}$ be an enumeration of $G$ such that $s_0=1_G$. Choose $(T_i)_{i\in\NN}$ a sequence of finite subsets of $G$ such that for every $i\in\NN$, $T_i\cap s_i T_i =\emptyset$ and $|T_i|=C\cdot i$, where $C$ is a constant to be defined later. These sets always exist as $G$ is infinite.
		
		Consider the uniform Bernoulli probability $\mu$ in $\{0,1\}^G$ and the collection of sets $\mathscr{A} := \{A_{n,g}\}_{n \geq 1, g \in G}$ given by $ A_{n,g} = \{ x \in \{0,1\}^G \mid x|_{gT_n} = x|_{gs_nT_n}\}$. Note that each set is a union of cylinders and that the existence of a configuration~$\tilde{x}$ in the intersection of the complement of these sets allows us to conclude the theorem by producing a configuration with the distinct neighborhood property. Our strategy is to apply Lemma~\ref{corollary.LLL_nonemptiness} to ensure its existence.
		
		As the intersection $s_nT_n \cap T_n$ is empty we have that $\mu(A_{n,g}) = 2^{-|T_n|} = 2^{-C n}$. Consider one set $A_{n,g}$. The number of sets $A_{m,g'}$ for a fixed $m \in \NN$ which are not independent from $A_{n,g}$ is at most $4C^2nm$ -- observe that $A_{n,g}$ is independent from $A_{m,g'}$if and only if $(gT_n\cup gs_nT_n)$ does not intersect $(g'T_m\cup g's_mT_m)$. We also define $x(A_{n,g}) := 2^{-\frac{Cn}{2}}$. Therefore, in order to conclude we must prove that: $$2^{-Cn} \leq 2^{-\frac{Cn}{2}} \prod_{m = 1}^{\infty}(1-2^{-\frac{Cm}{2}})^{4C^2nm}.$$
		
		Using the fact that $1-x \geq 2^{-2x}$ if $x \leq 1/2$ we obtain the following bound:
		
		\begin{align*}
		2^{-\frac{Cn}{2}} \prod_{m = 1}^{\infty}(1-2^{-\frac{Cm}{2}})^{4C^2nm}& \geq 2^{-\frac{Cn}{2}} \prod_{m = 1}^{\infty}2^{-8C^2nm2^{-\frac{Cm}{2}} } \\
		& = 2^{-\frac{Cn}{2}}2^{-8C^2 \sum_{m = 1}^{\infty}nm2^{-\frac{Cm}{2}} }
		\end{align*}

		Therefore, it suffices to prove that:
		\begin{align*}
		2^{-\frac{Cn}{2}} & \leq 2^{-8C^2 \sum_{m = 1}^{\infty}nm2^{-\frac{Cm}{2}} }\\
		\iff 1 & \geq  16C \sum_{m = 1}^{\infty}m2^{-\frac{Cm}{2}} \\
		\iff 1 & \geq  16C \frac{2^{\frac{C}{2}}}{(2^{\frac{C}{2}}-1)^2 } \\		
		\end{align*}
		
		The previous inequality holds true for $C \geq 17$. Therefore choosing $C = 17$ completes the proof by application of Lemma~\ref{corollary.LLL_nonemptiness}.\end{proof}

		\subsection{A graph-oriented construction and some computational properties}
		\label{subsection.strongly_aperiodic_subshifts_LLL}
		
		In this subsection we present another construction of a non-empty strongly aperiodic subshift. This construction is not as general as the previous one, as it only holds for finitely generated groups, and the size of the alphabet is rather large. Nevertheless, it can be defined by a natural property which allows us to use it in computability constructions with ease. 
		
		Let $\Gamma = (V,E)$ be a simple graph, consider a finite alphabet $\ag$ and a coloring $x \in \ag^{V}$ of the vertices of $\Gamma$. We say $x$ contains a \define{vertex-square path} if there exists an odd length path $p = v_1\dots v_{2n}$ such that $x_{v_i} = x_{v_{i+n}}$ for every $1 \leq i \leq n$. If the coloring $x$ does not contain any vertex-square path then we say it is a \define{square-free vertex coloring}. Next we show a proposition which is a slight modification of a proof which can be found in~\cor{\cite[Theorem~1]{alonetal_nonrepetitivecoloringofgraphs}}.
		
		\begin{proposition}\label{propositionsquarefreecoloringgraph}
			Let $G$ be a group which is generated by the finite set $S$. Then there exists a square-free vertex coloring of the undirected right Cayley graph $\Gamma(G,S)$ with $2^{19}|S|^2$ colors.
		\end{proposition}
		
		\begin{proof}
			Consider a finite alphabet $\ag$ and let $X = \ag^{\Gamma(G,S)}$ be the set of all vertex colorings of the Cayley graph $\Gamma(G,S)$. We define $\mu$ as the uniform Bernoulli probability, that is, for $a \in \ag$ and $g \in G$ then $$\mu( \{x \in X \mid x_g = a  \} ) = \frac{1}{|\ag|}.$$
			
			Consider $\mathcal{P}$ as the set of all odd length paths in $\Gamma(G,S)$. For $p \in \mathcal{P}$ let $A_{p}$ be the set of colorings of $\Gamma(G,S)$ such that $p$ is a square under that coloring and note that if $p$ is of length $2n-1$ then $\mu(A_p)= |\ag|^{-n}$ if there exists a path of said length. Consider $\ag_n= \{ A_p \mid p \text{ has length } 2n-1 \}$ and $\mathscr{A} = \{ A_p \mid p \in \mathcal{P}\} = \bigcup_{n \geq 1}\ag_n$. In order to apply Lemma~\ref{corollary.LLL_nonemptiness}, we define  $x(A_p) := (8|S|^2)^{-n}$ for $A_p \in \ag_n$. The lemma holds if for every $A \in \mathscr{A}$ then $\mu(A) \leq x(A)\prod_{B \in \Gamma(A)}(1-x(B))$. Replacing both sides yields the necessary condition:
			
			$$\forall n \geq 1, \ |\ag|^{-n} \leq (8|S|^2)^{-n}\prod_{j \geq 1}(1-(8|S|^2)^{-j})^{|\Gamma(A_p)\cap \ag_j|}.$$
			
			$|\Gamma(A_p)\cap A_j|$ corresponds to the amount of paths of length $2j-1$ which share a vertex with $p$. If $p$ has length $2n-1$ this can be bounded by $4nj(2|S|)^{2j}$. Indeed, there are at most $(2|S|)^{2j}$ paths of length $2j-1$ starting from a vertex $v$. Each of these paths can intersect a given vertex of $p$ in $2j$ positions and there are $2n$ vertices in $p$. \cor[]{Hence, we need to show: $$\forall n \geq 1, \ |\ag|^{-n} \leq (8|S|^2)^{-n}\prod_{j \geq 1}(1-(8|S|^2)^{-j})^{4nj(2|S|)^{2j}}.$$}

			 Using the inequality $1-x \geq 2^{-2x}$ if $x \leq 1/2$, the requirement to apply the lemma can be restrained further so that the following is required to conclude:
			
			\cor[]{\begin{align*}
				\forall n \geq 1, |\ag|^{-n} & \leq (8|S|^2)^{-n}\prod_{j \geq 1}2^{-8nj(8|S|^2)^{-j}(4|S|^2)^{j}}\\ & = (8|S|^2)^{-n}\prod_{j \geq 1}2^{-8nj2^{-j}}
				\end{align*}
				}
			
			\cor[E]{or e}quivalently:\begin{align*}
			|\ag| &\geq (8|S|^2) 2^{8\sum_{j \geq 1}{j2^{-j} } }\\
			& \geq 2^{19}|S|^2\cor{.}
			\end{align*} 
			
			\cor[Which]{The latter inequality} is satisfied by hypothesis, therefore, there exists a coloring of the graph such that no path of odd length is a square under that coloring.\end{proof}

		\begin{theorem}
			\label{theorem.strongly_aperiodic1}
			Every finitely generated group $G$ has a non-empty $G$-effectively closed strongly aperiodic subshift.
		\end{theorem}
		
		\begin{proof}
			Let $S$ be a set of generators of $G$ and $\ag$ an alphabet such that $|\ag| \geq 2^{19}|S|^2$. Consider the set of forbidden patterns $\FF$ defined as follows: Take $\mathcal{P}$ the set of all finite paths of odd length of $\Gamma(G,S)$. For every $p\in \mathcal{P}$ we define the set of patterns $\Pi_p$ as those with support $p$ and such that they are vertex-square paths. Let $\FF = \bigcup_{p\in \mathcal{P}}\Pi_p$ and let $X = X_{\FF}$ be the $G$-effectively closed subshift -- vertex-square paths can be recognized with a Turing machine with access to $\texttt{WP}(G)$ -- defined by this set of forbidden patterns. By Proposition~\ref{propositionsquarefreecoloringgraph} this subshift is non-empty. We claim it is strongly aperiodic.
			
			Let $x \in X$ and $g \in stab_{\sigma}(x)$. We are going to show that if $g \neq 1_G$ then $x$ contains a vertex-square path. Consider an expression of $g$ as an element of $(S\cup S^{-1})^*$ such it can be factorized as $g=_G uwv$ with $u =_G v^{-1}$. This can always be done by choosing $u=v=\varepsilon$ and $w$ a product of generators producing $g$. Amongst all those possible representations choose one such that $|w|$ is minimal. Clearly $|w| = 0$ implies that $g = 1_G$, so we suppose $|w|>0$. Let $w = w_1\dots w_n$ and consider the odd length walk $\pi = v_0v_1\dots v_{2n-1}$ defined by:
			
			$$v_i = \begin{cases}
			1_G & \text{ if } i = 0 \\
			w_1\dots w_i &\text{ if } i\in \{1,\dots, n \} \\
			ww_1\dots w_{i-n}& \text{ if } i\in \{n+1,\dots, 2n-1 \} \\
			\end{cases}$$
			
			We claim that $\pi$ is a path. Indeed, by definition, $w$ can not be reduced and thus there are no repeated vertices in $v_0v_1\dots v_n$ nor in $v_{n+1}\dots v_{2n-1}$. Therefore if there is a repeated vertex then it appears once in both parts. Suppose that it happens, thus we can consider two factorizations $w = ab$ and $w = cd$ such that $a =_G abc$. We obtain that $b = c^{-1}$. Obviously $|c| = |b|$, if not, $w$ can be written as $abcc^{-1}=_G ac^{-1}$ or $b^{-1}bcd=_G b^{-1}d$ which contradicts the minimality of $|w|$. Without loss of generality, we can replace $c$ by the word obtained by reversing the order and inversing the letters of $b$.  Moreover, $|c|>0$ and thus $|b| > 0$ which means that $w$ is written as follows:
			
			$$w = a_1\dots a_k b_1 \dots b_l = b_l^{-1} \dots b_1^{-1}d_1\dots d_k$$
			
			Therefore we can factorize $b_l^{-1}$ and $b_l$ from both sides obtaining a smaller word $w'$ in the representation of $g$. This contradiction shows that indeed $\pi$ is a path. We now show that it is a square-vertex path. As $g\in stab_{\sigma}(x)$, we also have $g^{-1} \in stab_{\sigma}(x)$ and thus for each $h \in G$ we have $x_{gh} = x_{h}$. Fix $h = uw_1\dots w_j$ for some $j \in \{1,\dots,n\}$, we obtain that $$x_{uw_1\dots w_j} = x_{uwu^{-1}uw_1\dots w_j} = x_{uww_1\dots w_j}.$$
			Applying $\sigma^{u}$ to both sides we obtain that $x_{w_1\dots w_j} = x_{ww_1\dots w_j}$ and therefore $x_{v_j} = x_{v_{j+n}}$, yielding a square-vertex path. Therefore $|w| = 0$ and thus $g = uv = 1_G$.
		\end{proof}
		
		Theorem~\ref{theorem.strongly_aperiodic1} provides a non-empty strongly aperiodic $G$-effectively closed subshift. Recently Jeandel~\cite{Jeandel2015} has shown that for recursively presented groups, if the group admits an effectively closed strongly aperiodic subshift then its word problem is decidable. Moreover, he has shown that the same conclusion stands when we allow every configuration to have a finite --instead of trivial-- stabilizer. Our result actually shows the converse, that is, that every recursively presented group with decidable word problem admits a non-empty strongly aperiodic effectively closed subshift. In the remainder of this section we prove this equivalence.
		
		\begin{lemma}\label{compacityfunctionaperiodic}
			Let $G$ be a finitely generated group and $X \subset \ag^G$ a non-empty strongly aperiodic subshift. There exists a function $f: \NN \to \NN$ such that for every $x \in X$ \cor[then]{, if} $g \in B(1_G,n) \setminus \{1_G\}$ \cor[$\implies$]{then} $x|_{B(1_G,f(n))} \neq x|_{B(g,f(n))}$.
		\end{lemma}
		
		\begin{proof}
			Suppose $f$ does not exist, thus there exists $n \in \NN$ and a sequence $(x_j,g_j)_{j \in \NN} \subset X\times  B(1_G,n) \setminus\{1_G\}$ such that $x_j|_{B(1_G,j)} = x_j|_{B(g_j,j)}$. As $B(1_G,n)$ is finite there exists $\bar{g} \neq 1_G$ which appears infinitely often in $(g_j)_{j \in \NN}$. Consider the subsequence $(x_k)_{k \in \NN, g_{k} = \bar{g}}$ and from there extract a  converging subsequence $(x_{k_{\alpha}}) \to \bar{x} \in X$. We claim that $\bar{g}^{-1} \in stab_{\sigma}(\bar{x})$. By definition of convergence, for every $m \in \NN$ there exists $N_{\alpha} \geq m$ such that $\bar{x}|_{B(1_G,m+n)} = (x_{N_{\alpha}})|_{B(1_G,m+n)}$ and thus $$\bar{x}|_{B(1_G,m)} = (x_{N_{\alpha}})|_{B(1_G,m)} = (x_{N_{\alpha}})|_{B(\bar{g},m)} = \bar{x}|_{B(\bar{g},m)}$$
			
			So for every $m \in \NN$ we have $\bar{x}|_{B(1_G,m)} = \bar{x}|_{B(\bar{g},m)}$ and therefore $\forall g \in G: \bar{x}_g = \bar{x}_{\bar{g}g}$. Which yields a contradiction as $X$ is strongly aperiodic. \end{proof}
		
		\begin{theorem}
			Let $G$ be a recursively presented and finitely generated group. There exists a non-empty strongly aperiodic effectively closed subshift if and only if $\texttt{WP}(G)$ is decidable.
		\end{theorem}
		
		\begin{proof}
			Every $G$-effectively closed subshift is effectively closed when $\texttt{WP}(G)$ is decidable. Therefore Theorem~\ref{theorem.strongly_aperiodic1} yields the desired construction. Conversely, suppose there is a non-empty effectively closed subshift $X$ which is strongly aperiodic. As $G$ is recursively presented then $\texttt{WP}(G)$ is recognizable. Let $T$ be a Turing machine which accepts every inconsistent pattern coding and a maximal set of consistent pattern codings which generates $\FF$ such that $X = X_{\FF}$. The existence of such a machine in the case of a recursively presented group is given in~\cite{ABS2014}. 
			
			Let $w \in (S \cup S^{-1})^*$. We present an algorithm which accepts if and only if $w \neq_G 1_G$. Consider the ball of size $n$ in the free monoid over the alphabet $(S \cup S^{-1})^*$, that is $\Lambda_n = \{u \in (S \cup S^{-1})^*| |u| \leq n \}$ and consider the set $\Lambda_n \cup w\Lambda_n$. For each one of these sets we construct the set $\Pi_n$ of all pattern codings $c$ such that for $u \in \Lambda_n$ then $(u,a) \in c$ if and only if $(wu,a) \in c$. That is, we force the ball of size $n$ around the empty word $\epsilon$ and $w$ to be the same. Consider the algorithm which iteratively runs $T$ on every pattern coding of $\Pi_1, \Pi_2, \dots \Pi_j$ up to $j$ steps and then does $j \leftarrow j+1$ and which accepts $w$ if and only if every pattern coding in a particular $\Pi_i$ is accepted by $T$. If $w =_G 1_G$ the algorithm can never accept as it would mean no patterns are constructible around $1_G$ and thus $X = \emptyset$. Conversely, if $w \neq_G 1_G$ then using the function $f$ given by 
			Lemma~\ref{compacityfunctionaperiodic} we get that for every $x \in X$ if $w \neq_G 1_G$ then $x|_{B(1_G,f(|w|))} \neq x|_{B(w,f(|w|))}$ thus every pattern in $\Pi_{f(|w|)}$ is either inconsistent or represents a forbidden pattern, and therefore $T$ must accept every pattern of $\Pi_{f(|w|)}$. \end{proof}
		
		One may ask if it is possible to construct non-empty strongly aperiodic subshifts which satisfy stronger constrains, such a being of finite type, sofic or effectively closed. The previous result shows that our construction is in this sense optimal for recursively presented groups with undecidable word problem.
	
	\section{Realization of densities}
	\label{section.realization_densities}
	
	In this section we construct over any infinite and finitely generated group a non-empty subshift over $\{0,1\}$ with the property that the density of $1$'s over any F\o lner sequence converges to a fixed $\alpha \in [0,1]$. From this result we derive the existence of uniform density subshifts for infinite groups of subexponential growth for any finite set of generators. Furthermore, we show that said subshifts are always weakly aperiodic.
	
	\begin{definition}
		Let $F \subset G$ be a finite subset of a group and $x \in \{0,1\}^G$ be a configuration. We define the \emph{density of $1$'s in $F$ and $x$} as: $$\dens(1,x|_F) = \dens(1,F,x) = \frac{|\{g \in F \mid x_g = 1 \}|}{|F|}.$$
	\end{definition}
	
	Similarly if $P\in\ag^F$ is a pattern with support $F$, we denote by $\dens(1,P)$ the ratio $\frac{|\{g \in F \mid P_g = 1 \}|}{|F|}$.
	
	\begin{definition}
		Let $G$ be a finitely generated group and $S$ a finite set of generators. We say a $G$-subshift over $\{0,1\}$ has uniform density $\alpha \in [0,1]$ if for every configuration $x \in X$ and for every sequence $(g_n)_{n\in\NN}$ of elements in $G$, one has $\dens(1,B(g_n,n),x)\to \alpha$.
	\end{definition}
	
	In a way similar to the previous definition, we could say a configuration $x \in \{0,1\}^G$ has density $\alpha \in [0,1]$ for some sequence of subsets $(T_n)_{n \in \NN}$ if for each sequence of elements $(g_n)_{n \in \NN}$ we have that $\dens(1,g_nT_n,x) \to \alpha$. Nevertheless, contrary to the preceding section, Lov\'asz local lemma cannot directly be applied to prove the existence of \cor{such} configurations. If we define the forbidden sets to be $A_{n,g} = \{x \in \{0,1\}^G \mid |\dens(1,gT_n,x)- \alpha | > \delta_n\alpha \}$ for some sequence of error terms $\delta_n \to 0$ we obtain that the measure of this set can be bounded \cor[by]{} above using the Chernoff bounds by $2\exp(\delta_n^2\alpha|T_n|/3)$. For any function which bounds these values \cor[by]{} above, and after some elimination of exponents, we obtain that the left hand side of the inequality required by the local lemma depends on $\delta_n$ while the right hand side is constant. Therefore we tackle this problem with a different 
approach.
	
	Nevertheless, if the condition that the group is amenable is added, not only it is possible to obtain a result like the one defined in the previous paragraph, moreover, the density over every F\o lner sequence can be asked to converge to the same fixed $\alpha$.
	
	\cor[!! this paragraph has been moved !!]{}
	A group $G$ is called \define{amenable} if there exists a left-invariant finitely additive probability measure $\mu: \mathcal{P}(G) \to [0,1]$ on $G$. The amenability of a group has many equivalent definitions -- many of which can be found in~\cite{ceccherini-SilbersteinC09}. From a combinatorial point of view the F\o lner condition states that a group is amenable if and only if it admits a F\o lner net, that is, a net $F_{\alpha}$ of non-empty finite subsets of $G$ such that $\forall g \in G$: $$ \lim_{\alpha}{ \frac{|gF_{\alpha} \triangle F_{\alpha}|}{|F_{\alpha}|}}= 0.$$
		
	Let $Int(F,K) := \{g \in F | \forall k \in K, gk \in F\}$ be the interior of $F$ with respect to $K$ and $\partial_K F := F \setminus Int(F,K)$ the boundary of $F$ with respect to $K$. In the case of countable groups the net can be just taken to be a sequence and thus amenability can be shown to be equivalent to the fact for every finite $K \subset G$ we have $\lim_{n \to \infty}{\frac{|\partial_K F_n|}{|F_n|}}=0$. That is to say, for any finite set $K$ the boundaries of the sets $F_n$ with respect to $K$ grow slower than themselves.
	
	\begin{theorem}
		\label{theorem.densities}
			Let $G$ be an infinite and finitely generated \cor{amenable} group and $\alpha \in [0,1]$. \cor[There is a non-empty subshift $X_{\alpha} \subset \{0,1\}^G$ such that for any $x \in X_{\alpha}$ and F\o lner sequence $(F_n)_{n \in \NN}$ then $\lim_{n \to \infty} \dens(1,F_n,x) = \alpha$]{Then there exists a non-empty subshift $X_{\alpha} \subset \{0,1\}^G$ with $\lim_{n \to \infty} \dens(1,F_n,x) = \alpha$ for any $x \in X_{\alpha}$ and any F\o lner sequence $(F_n)_{n \in \NN}$}.
	\end{theorem}
	
	Before proving this theorem we \cor[give a brief introduction to amenable groups in order to fix the notations and ]{} prove a useful property of metric spaces.

	\begin{definition}
		Let $(X,d)$ be a metric space. We say $F \subset X$ is \define{$r$-covering} if for each $x \in X$ there is $y \in F$ such that $d(x,y) \leq r$. We say $F$ is \define{$s$-separating} if for each $x \neq y \in F$ then $d(x,y) > s$.
	\end{definition}
	
	\begin{figure}[!ht]
		\centering
		\begin{tikzpicture}[scale=1]
 
%

\draw []
[l-system={cayleyPSL2Z, step=4pt, angle=30, axiom=[F+FFF[+FA]----FFF[+FA]----FFF]++++++A, order=4}]
lindenmayer system -- cycle;
\draw[fill=vert] (0.15,0) circle (0.05);
\draw[fill=vert] (-1.65,-0.63) circle (0.05);
\draw[fill=vert] (-1.65,0.63) circle (0.05);
\draw[fill=vert] (-2.42,0.63) circle (0.05);
\draw[fill=vert] (-2.5,1.3) circle (0.05);
\draw[fill=vert] (-2.42,-0.9) circle (0.05);
\draw[fill=vert] (-2.5,-0.23) circle (0.05);
\draw[fill=vert] (-1.6,1.6) circle (0.05);
\draw[fill=vert] (-1.1,-1.3) circle (0.05);
\draw[fill=vert] (-1.1,1.3) circle (0.05);
\draw[fill=vert] (-1.6,-1.8) circle (0.05);

\draw[fill=vert] (1.15,1.6) circle (0.05);
\draw[fill=vert] (0.3,2.22) circle (0.05);
\draw[fill=vert] (-0.22,1.8) circle (0.05);
\draw[fill=vert] (-0.22,2.62) circle (0.05);
\draw[fill=vert] (1,2.9) circle (0.05);
\draw[fill=vert] (1.62,2.9) circle (0.05);
\draw[fill=vert] (2.2,1.6) circle (0.05);
\draw[fill=vert] (2.25,2.25) circle (0.05);
\draw[fill=vert] (2.25,0.65) circle (0.05);
\draw[fill=vert] (2.72,1.97) circle (0.05);
\draw[fill=vert] (2.72,0.92) circle (0.05);

\draw[fill=vert] (1.15,-1.6) circle (0.05);
\draw[fill=vert] (0.3,-2.22) circle (0.05);
\draw[fill=vert] (-0.22,-1.8) circle (0.05);
\draw[fill=vert] (-0.22,-2.62) circle (0.05);
\draw[fill=vert] (1,-2.9) circle (0.05);
\draw[fill=vert] (1.62,-2.9) circle (0.05);
\draw[fill=vert] (2.2,-1.6) circle (0.05);
\draw[fill=vert] (2.25,-2.25) circle (0.05);
\draw[fill=vert] (2.25,-0.65) circle (0.05);
\draw[fill=vert] (2.72,-1.97) circle (0.05);
\draw[fill=vert] (2.72,-0.92) circle (0.05);

\end{tikzpicture}
		\caption{In green, an example of $2$-covering and $2$-separating set in $PSL(2,\ZZ) \cong \ZZ/2\ZZ \ast \ZZ/3 \ZZ$. Green vertices are at distance at least 3 from each other, and every vertex is at distance at most 2 from a green vertex.}
		\label{figure.hierarchical_decomposition_PSL2Z_bis}
	\end{figure}
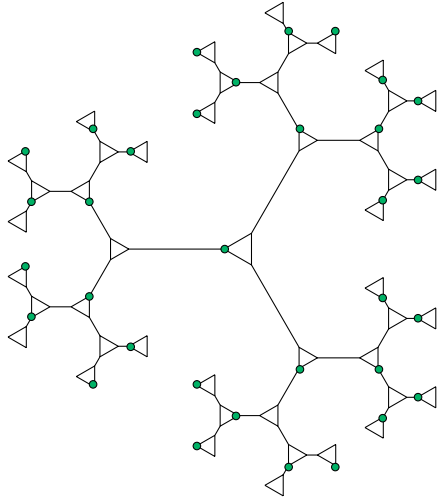
		
	\begin{lemma}\label{Lemmametricspacedecomposition}
		Let $(X,d)$ be a metric space and $r \in \NN$. There exists a set $F_r \subset X$ which is $r$-separating and $r$-covering. 
	\end{lemma}
	
	\begin{proof}
		Suppose we have \cor[a]{an} $r$-separating set $F$ which is not $r$-covering. Then the set $K := \{x \in X \mid d(F,x)>r\}$ is not empty and $F$ can be extended by an element of $K$. Thus any maximal $r$-separating set is also $r$-covering.
		
		We only need to show that maximal $r$-separating sets exist. Let $(\mathcal{S},\subset)$ be the set of $r$-separating subsets of $X$ ordered by inclusion. Clearly $\emptyset \in \mathcal{S}$ and given a chain $\{A_i\}_{i \in I} \subset \mathcal{S}$ we have that $A = \bigcup_{i \in I} A_i$ is an upper-bound. Indeed, if $x,y \in A$ then $x \in A_i$ and $y \in A_j$ for some $i,j \in I$. As $\{A_i\}_{i \in I}$ is a chain, then either $A_i \subset A_j$ or $A_j \subset A_i$. As any of these two sets is $r$-separating we get that $d(x,y)>r$ and hence $A \in \mathcal{S}$. By Zorn's lemma there exists a maximal $r$-separating set. \cor[ This is false, the limit of $\chi$ is separated, but not necessarily maximal!!!, if $X = \mathbb{Q}$, then $B(x_0,n)$ is infinite, how do you pick your maximal? Another example: $(\NN,d)$ for $d(a,b) = 1$ except if $a=b$. A $2$-separated maximal set is any singleton, but the sequence $F_i = \{i\}$ converges to the empty set (which is separated but not maximal).The existence of a 
maximal $r$-separating set for the case of a countable metric space is given by the sequential compactness of $\{0,1\}^X$. Fix some element $x_0\in X$. Consider $F_n$ a maximal $r$-separating set of the ball $B(x_0,n)$. These finite sets can be enumerated as a sequence $(\chi_n)_{n\in\NN}$ in $\{0,1\}^X$, where $\chi_n\in\{0,1\}^X$ is the characteristic function of $F_n$. By compactness of $\{0,1\}^X$, there exists a subsequence $\left(F_{\phi(n)}\right)_{n\in\NN}$ that converges to some $y\in\{0,1\}^X$. By definition of the distance $d$, these characteristic functions correspond to finite sets $F_{\phi(n)}$ that coincide on larger balls $B(x_0,n)$ as $n$ goes to infinity. Denote $H$ the subset of $X$ with characteristic function $y$. Any two elements $x,y$ of $H$ belong to a finite set $F_n$ of the aforementioned sequence, thus if $H$ is not $s$-separating, this $F$ also fails to $r$-separates $X$, which contradicts the definition of $H$. Thus $H$ is $r$-separating. Again by contradiction, if $H$ were not 
maximal, there would exist an element $x\in X\setminus H$ such that $H\cup\{x\}$ is $r$-separating. Then $H\cap B(x_0,|x|)$ would not be maximal, which contradicts the definition of the $F_n$. As each finite subset $F \subset X$ can be seen as its characteristic function in $\{0,1\}^X$.]{}
	\end{proof}
	
	No\cor{w} we are ready to begin the proof of Theorem~\ref{theorem.densities}.
	
	\begin{proof}
		
		If $\alpha \in \{0,1\}$ the result is trivial. Let $\alpha \in (0,1)$, and define $K_n := B(1_G,5^n)$ and consider the subshift $X_{\alpha}$ given by the set of forbidden patterns $\FF$ such that for $P \in \{0,1\}^F$ (where $F \subset G$, $|F|<\infty$) belongs to $\FF$ if and only if the following condition is not satisfied:
		
		$$2n|\partial_{K_n}F| < |F| \implies |\dens(1,P)-\alpha| \leq \frac{1}{n}$$
		
		In other words, we forbid a pattern $P$ with support $F$ if the ratio $\frac{|\partial_{K_n}F|}{|F|}$ is sufficiently small and the density of ones in $P$ is further than $\frac{1}{n}$ from $\alpha$.
		
		Consider a F\o lner sequence $(F_n)_{n \in \NN}$ and let $m \in \NN$ and $x \in X_{\alpha}$. As $\lim_{n \to \infty}\frac{|\partial_{K_m}F_n|}{|F_n|} = 0$ there exists $N \in \NN$ such that $$\forall n \geq N \ \ \ \ \ \frac{|\partial_{K_m}F_n|}{|F_n|} < \frac{1}{2m}$$
		
		Therefore, for every $n \geq N$ we get that $|\dens(1,x|_{F_n})-\alpha| \leq \frac{1}{m}$. As $m$ can be made arbitrarily big we obtain that $\lim_{n \to \infty} \dens(1,F_n,x) = \alpha$.
		
		\bigskip
		
		We only need to show that $X_{\alpha} \neq \emptyset$. Our strategy will be to inductively construct an infinite covering forest of $G$, and then put a Sturmian word along an enumeration of the leaves of each of its trees. The configuration $x\in\{0,1\}^G$ obtained by this process will belong to $X_\alpha$. The following objects -- that are formally described below -- will be used to formalize this idea: a sequence $(A_n)_{n \in \NN} \subset 2^G$ of subsets of $G$, a sequence $(p_n)_{n \in \NN}: G \to A_n$ of functions and a sequence $(\Gamma_n)_{n \in \NN}$ of graphs on vertex sets $(A_n)_{n \in \NN}$ respectively. They are defined by the following recurrences, with base cases $A_0 = G$, $p_0 = id$ and $\Gamma_0 = \Gamma(G,S)$ where $S$ is a finite set of generators of $G$: 
		
		\label{inductive_structure_forest}
		
		\begin{enumerate}
			\item  The set $A_{n+1}$ is chosen as a $2$-separating and $2$-covering subset of $A_n$ for the distance induced by $\Gamma_n$. In particular, the sets $(A_n)_{n \in \NN}$ are nested.
			\item Suppose $p_n:G\to A_n$ is already defined, we first define $p_{n+1}$ on $A_{n}$ and then extend it to $G$. Consider an element $g\in A_n$. Since $A_{n+1}$ 2-covers $A_n$ in $\Gamma_n$, there are only three possible cases.
			\begin{itemize}
				\item $g\in A_{n+1}$: in this case we set $p_{n+1}(g)=g$.
				\item $d_{\Gamma_n}(g,A_{n+1})=1$: there exists a unique $h\in A_{n+1}$ that satisfies $d_{\Gamma_n}(g,h)=1$ -- uniqueness comes from the fact that $A_{n+1}$ is 2-separating -- and we set $p_{n+1}(g)=h$.
				\item $d_{\Gamma_n}(g,A_{n+1})=2$: we arbitrarily choose one $h\in A_{n+1}$ that realizes $d_{\Gamma_n}(g,h)=2$ and set $p_{n+1}(g)=h$.
			\end{itemize}
			For $g' \in G\setminus A_n$ we finally extend this function by $p_{n+1} := p_{n+1}\circ p_{n}$. 
			\item For $g \in A_n$ define the \define{$n$-cluster of $g$} by $\CC_n(g) := \{h \in G \mid p_{n}(h)=g \}$. The element $g\in A_n$ is called the \define{center} of the cluster $\CC_n(g)$. The graph $\Gamma_{n+1}$ has vertex set $A_{n+1}$, and there is an edge in $\Gamma_{n+1}$ between two elements $g,h \in A_{n+1}$ if and only if there exist $g'\in \CC_n(g)$ and $h'\in \CC_n(h)$ that are neighbors in $\Gamma(G,S)$.
		\end{enumerate}
		
		The \define{covering forest} defined by the sequence $(A_n,p_n,\Gamma_n)_{n \in \NN}$ is $(V,E)$, where the set of vertices  $V$ is the multiset $\bigsqcup_{n\in\NN} A_n$, and the edges are given by the parent functions: $(g,h)\in E$ if and only if $g\in A_n$, $h\in A_{n+1}$ and $p_n(g)=h$. In particular the successive applications of $p_1,\dots,p_n$ to an element $g\in G=A_0$ gives the path from the leaf labeled by $g$ to its height $n$ parent. The cluster $\CC_n(g)$ corresponds to the set of labels of descendants of the node labeled by $g$ that appears at height $n$ in the covering forest. The cluster $\CC_{n+1}(g)$ is obtained as the union of the cluster $\CC_n(g)$, all clusters $\CC_n(h)$ for $h\in A_n$ such that $d_{\Gamma_n}(g,h)=1$ and clusters $\CC_n(h')$ for $h'\in A_n$ such that $d_{\Gamma_n}(g,h')=2$ for which the parent function $p_{n+1}(h')$ has been chosen to be $g$ (see Figure~\ref{figure.covering_tree}). Remark that every cluster $\CC_n(g)$ is connected in $\Gamma$ as it is the finite union 
		of adjacent connected sets in $\Gamma$.

		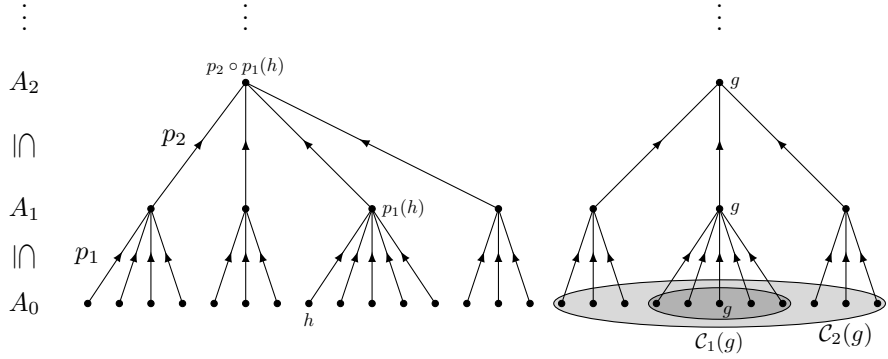
\begin{figure}[!ht]
			\centering
			\begin{tikzpicture}[scale=0.42]

\draw[fill=black!15] (21,0) ellipse (5.25 and 0.75);
\draw[fill=black!30] (21,0) ellipse (2.25 and 0.5);

\draw (21.25,-0.25) node{\scalebox{0.7}{$g$}};
\draw (21.5,3) node{\scalebox{0.7}{$g$}};
\draw (21.5,7) node{\scalebox{0.7}{$g$}};

\draw (8,-0.5) node{\scalebox{0.7}{$h$}};
\draw (11,3) node{\scalebox{0.7}{$p_1(h)$}};
\draw (6,7.5) node{\scalebox{0.7}{$p_2\circ p_1(h)$}};

\draw (21,-1.25) node{\scalebox{0.8}{$\CC_{1}(g)$}};
\draw (25,-1) node{\scalebox{0.9}{$\CC_{2}(g)$}};

\foreach \x in {1,...,26}{
  \draw[fill=black] (\x,0) circle (0.1);
}
\draw (-1,0) node{$A_0$};
\draw (1,1.5) node{$p_1$};

\foreach \x in {3,6,10,14,17,21,25}{
  \draw[fill=black] (\x,3) circle (0.1);
}
\draw (-1,1.5) node[yscale=1.5,rotate=-90]{$\subseteq$};
\draw (-1,3) node{$A_1$};
\draw (3.75,5.25) node{$p_2$};

\foreach \x in {1,...,4}{
  \draw[thin,postaction={decorate,decoration={markings,
    mark=at position .55 with {\arrow[scale=1]{latex}}}}] (\x,0) -> (3,3);
} 
\foreach \x in {5,...,7}{
  \draw[thin,postaction={decorate,decoration={markings,
    mark=at position .55 with {\arrow[scale=1]{latex}}}}] (\x,0) -> (6,3);
}  
\foreach \x in {8,...,12}{
  \draw[thin,postaction={decorate,decoration={markings,
    mark=at position .55 with {\arrow[scale=1]{latex}}}}] (\x,0) -> (10,3);
} 
\foreach \x in {13,...,15}{
  \draw[thin,postaction={decorate,decoration={markings,
    mark=at position .55 with {\arrow[scale=1]{latex}}}}] (\x,0) -> (14,3);
}  
\foreach \x in {16,...,18}{
  \draw[thin,postaction={decorate,decoration={markings,
    mark=at position .55 with {\arrow[scale=1]{latex}}}}] (\x,0) -> (17,3);
} 
\foreach \x in {19,...,23}{
  \draw[thin,postaction={decorate,decoration={markings,
    mark=at position .55 with {\arrow[scale=1]{latex}}}}] (\x,0) -> (21,3);
} 
\foreach \x in {24,...,26}{
  \draw[thin,postaction={decorate,decoration={markings,
    mark=at position .55 with {\arrow[scale=1]{latex}}}}] (\x,0) -> (25,3);
} 

\foreach \x in {6,21}{
  \draw[fill=black] (\x,7) circle (0.1);
}
\draw (-1,5) node[yscale=1.5,rotate=-90]{$\subseteq$};
\draw (-1,7) node{$A_2$};

\foreach \x in {3,6,10,14}{
  \draw[thin,postaction={decorate,decoration={markings,
    mark=at position .55 with {\arrow[scale=1]{latex}}}}] (\x,3) -> (6,7);
} 
\foreach \x in {17,21,25}{
  \draw[thin,postaction={decorate,decoration={markings,
    mark=at position .55 with {\arrow[scale=1]{latex}}}}] (\x,3) -> (21,7);
} 

\draw (-1,9) node[rotate=-90]{$\dots$};
\draw (6,9) node[rotate=-90]{$\dots$};
\draw (21,9) node[rotate=-90]{$\dots$};

\end{tikzpicture}
			\caption{A covering forest of $G$. In the left section of the image the edge structure is emphasized by writing explicitly the parent functions. In the right section we remark the cluster structure for $g \in A_2$.}
			\label{figure.covering_tree}
		\end{figure}
		
		Note that definition 3 above is equivalent to what follows: for $g,h \in A_{n+1}$ then the edge $(g,h)$ is in $E(\Gamma_{n+1}) $ if and only if there exists a path $g_1=g,g_2, \dots ,g_m=h$ from $g$ to $h$ in $\Gamma(G,S)$ such that for every $i \in \{1,\dots,m\}$ we have $p_{n+1}(g_i) \in \{g,h\}$.
		
		\begin{claim}\label{claimbolas}
			Let $g \in A_n$, then $B(g,n) \subset \CC_n(g) \subset B(g,\frac{1}{2}(5^n-1))$.
		\end{claim}
		\begin{proof}
			We prove the claim by induction. It stands true for $n = 0$.
			\begin{itemize}
				\item Consider $\CC_n(g)$. By induction hypothesis, $B(g,n-1) \subset \CC_{n-1}(g) \subset \CC_{n}(g)$. Let $h \in B(g,n)\setminus B(g,n-1)$. Either $h \in C_{n-1}(g)$ and we are done, or $h \in C_{n-1}(g')$ for some $g'\in A_{n-1}$. Then necessarily $d_{\Gamma_{n-1}}(g,g')=1$, since $hs\in B(g,n-1) \subset \CC_{n-1}(g)$ for some $s\in S\cup S^{-1}$. Finally as $A_{n}$ is a $2$-separating subset of the vertices of $\Gamma_{n-1}$ we get that $\CC_{n-1}(g') \subset \CC_{n}(g)$ thus $h \in \CC_{n}(g)$. We conclude that $B(g,n) \subset \CC_n(g)$. Note that the same argument proves that $\CC_{n}(g')\cdot(S\cup S^{-1})\subset \CC_{n+1}(g')$ for every $n\in\NN$ and $g' \in A_{n+1}$.
				
				\item Suppose inductively that for every $g \in A_{n-1}$ the inclusion $\CC_{n-1}(g) \subset B(g,\frac{1}{2}(5^{n-1}-1))$ holds. Fix one $g\in A_n$ and consider an element $h$ in the cluster $\CC_{n}(g)$. We show that $d_G(h,g)\leq \frac{1}{2}(5^{n}-1)$ by constructing a path of length at most $\frac{1}{2}(5^{n}-1)$ from $h$ to $g$. By definition of the cluster $\CC_n(g)$, we know that the element $h'\in A_{n-1}$ such that $h\in\CC_{n-1}(h')$ satisfies $d_{\Gamma_{n-1}}(g,h')\leq2$. In the sequel we will only consider the case where this distance is exactly $2$ as it is the worst case. Thus we assume that there exists a path $h'\to h''\to g$ of length $2$ between this $h'$ and $g$ in $\Gamma_{n-1}$. By definition of the graph $\Gamma_{n-1}$, this implies the existence of $k'\in\CC_{n-1}(h')$ and $k''\in\CC_{n-1}(h'')$ that are neighbors in $\Gamma(G,S)$ and $\ell''\in\CC_{n-1}(h'')$ and $\ell\in\CC_{n-1}(g)$ that are neighbors in $\Gamma(G,S)$. Putting everything together, we can build the following path 
in	$\Gamma(G,S)$ (see Figure~\ref{figure.upper_bound_cluster}): $$ h \to \dots \to h'\to \dots \to k' \to k'' \to \dots \to h''\to \dots \to \ell''\to \ell \to \dots \to g. $$
				
				\begin{figure}[!ht]
					\centering
					\begin{tikzpicture}[scale=1]

\draw (-1.5,0.75) node[above right]{$h$};
\draw[fill=black] (-1.5,0.75) circle (0.05);
 
\draw[dashed] (-1.5,0.75) -- (-1.25,0.5) -- (-1,0.75) -- (-0.5,0.5) -- (0,0); 
 
\draw[thick] (0,0) circle (2); 
\draw (0,0) node[below right]{$h'$};
\draw[fill=black] (0,0) circle (0.05);

\draw (0,-2.5) node[]{$\CC_{n-1}(h')$};

\draw[dashed] (0,0) -- (0.5,0.5) -- (1,0.25) -- (1.25,-0.25) -- (1.75,0);

\draw (1.75,0) node[above left]{$k'$};
\draw[fill=black] (1.75,0) circle (0.05);
\draw[] (1.75,0) -- (2.25,0);
\draw (2.25,0) node[above right]{$k''$};
\draw[fill=black] (2.25,0) circle (0.05);

\draw[dashed] (2.25,0) -- (2.5,-0.25) -- (3,-0.75) -- (3.25,-0.25) -- (3.5,0);

\draw[thick] (3.5,0) circle (1.5); 
\draw (3.5,0) node[below right]{$h''$};
\draw[fill=black] (3.5,0) circle (0.05);

\draw (3.5,-2) node[]{$\CC_{n-1}(h'')$};

\draw[dashed] (3.5,0) -- (3.75,0.5) -- (4,0.25) -- (4.5,-0.25) -- (4.75,0);

\begin{scope}[shift={(3,0)}]
\draw (1.75,0) node[above left]{$\ell''$};
\draw[fill=black] (1.75,0) circle (0.05);
\draw[] (1.75,0) -- (2.25,0);
\draw (2.25,0) node[above right]{$\ell$};
\draw[fill=black] (2.25,0) circle (0.05);
\end{scope}

\draw[dashed] (5.25,0) -- (5.5,-0.5) -- (6,0.25) -- (6.25,-0.25) -- (6.75,0);

\draw[thick] (6.75,0) circle (1.75); 
\draw (6.75,0) node[below right]{$g$};
\draw[fill=black] (6.75,0) circle (0.05);

\draw (6.75,-2.25) node[]{$\CC_{n-1}(g)$};
 
\end{tikzpicture}
					\caption{A path from an element $h$ of $\CC_n(g)$ to $g$ which inductively proves the inclusion $\CC_{n}(g) \subset B(g,\frac{1}{2}(5^{n}-1))$.}
					\label{figure.upper_bound_cluster}
				\end{figure}
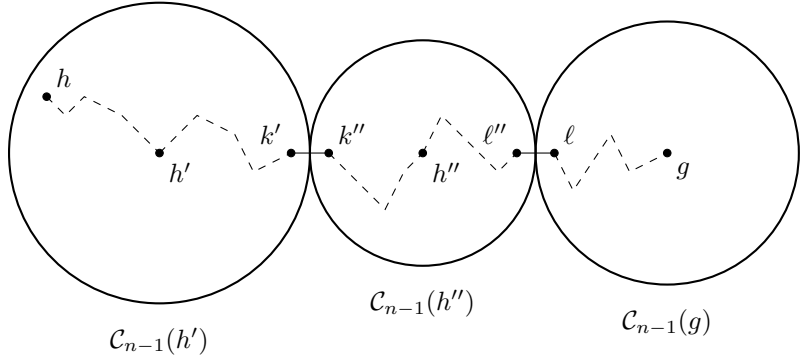

				Since they all link an element of a cluster of level $n-1$ to the center of this cluster, the induction hypothesis implies that we can choose the five subpaths $h \to \dots \to h'$, $h'\to \dots k'$, $k'' \to \dots \to h''$, $h''\to \dots \to \ell''$ and $\ell \to \dots \to g$ of length at most $\frac{1}{2}(5^{n-1}-1)$. Thus the total length of the path is at most $5\cdot\frac{1}{2}(5^{n-1}-1)+2\leq\frac{1}{2}(5^{n}-1)$. Therefore $\CC_{n}(g) \subset B(g,\frac{1}{2}(5^{n}-1))$.
				
			\end{itemize}
		\end{proof}
		
		Let $x \in \{0,1\}^G$ be a configuration such that for every $n \in \NN$ and $g \in A_n$
		
		\begin{align}
			\lfloor \alpha|\CC_{n}(g)| \rfloor \leq |\{h \in \CC_n(g) \mid x_h = 1\}| \leq \lfloor \alpha|\CC_{n}(g)| \rfloor +1 
			\label{condition.good_alpha}.
		\end{align}
		
		\begin{claim}\label{claim.exists_config_condition}
			There exists a configuration $x$ that satisfies condition~(\ref{condition.good_alpha}).
		\end{claim}
		
		\begin{proof}
			Consider the covering forest given by some sequence $(A_n,p_n,\Gamma_n)_{n \in \NN}$ as specified above. For every component $C$ of this forest, take $\phi_C$ a convex enumeration of its leaves: if $g$ and $g'$ are two leaves of $C$ with the same parent of height $n$ for some $n\in\NN$ -- i.e. $p_n(g)=p_n(g')$ -- then the preimage $h$ of every integer between $\phi_C(g)$ and $\phi_C(g')$ satisfies that $p_n(h)=p_n(g)$. Such an enumeration always exists.
			
			Let $(w_k)_{k \in \NN}$ be a Sturmian word of slope $\alpha$. We can build a configuration $x$ by putting the Sturmian sequence $(w_k)_{k \in \NN}$ along the convex enumeration chosen for every component of the forest. Since Sturmian words are balanced, we deduce that the configuration $x$ satisfies condition~(\ref{condition.good_alpha}). \end{proof}
		
		\begin{claim}\label{claim.condition_implies_alpha}
			If a configuration $x\in\{0,1\}^G$ satisfies condition~(\ref{condition.good_alpha}), then $x$ belongs to $X_{\alpha}$.
		\end{claim}
		\begin{proof}
			Let $x$ be such a configuration and take some $n\in\NN$. Let $F$ be a set such that $2n|\partial_{K_n}F| < |F|$ -- remember that $K_n$ is $B(1_G,5^n)$ -- and consider the pattern $P := x|_{F}$. Let $V := Int(F,B(1_G,\frac{1}{2}(5^n-1))) \cap A_n$ and $R = F \setminus \bigcup_{v \in V}\CC_n(v)$. As $\bigcup_{v \in V}\CC_n(v) \subset  \bigcup_{v \in V}vB(1_G,\frac{1}{2}(5^n-1)) \subset F$ we get that:
			
			$$\frac{1}{|F|}\sum_{v \in V}(\lfloor \alpha|\CC_{n}(v)| \rfloor) \leq  \dens(1,P) \leq  \frac{1}{|F|}\sum_{v \in V}(\lfloor \alpha|\CC_{n}(v)| \rfloor +1) + \frac{|R|}{|F|}.$$
			
			Before working on those inequalities we remark two facts:
			
			\begin{enumerate}
				\item $R \subset \partial_{K_n}F$. Therefore $\frac{|R|}{|F|} < \frac{1}{2n}$.
				\item $ |V| \leq \frac{|F|}{|B(1_G,n)|}$.
			\end{enumerate}
			Indeed, let $r \in Int(F,K_n)$. That is, for all $g \in K_n$ then $rg \in F$. As $d(r,p_n(r)) \leq \frac{1}{2}(5^n-1)$ then $p_n(r) \in Int(F,B(1_G,5^n-\frac{1}{2}(5^n-1))) \subset Int(F,B(1_G,\frac{1}{2}(5^n-1)))$ therefore $p_n(r) \in V$. That means that $r \notin R$, therefore $R \subset F \setminus Int(F,K_n) = \partial_{K_n}F$. The second remark is a consequence of Claim~\ref{claimbolas}.	
			
			From the left side we get:\begin{align*}
				\dens(1,P)&  \geq  \frac{1}{|F|}\sum_{v \in V}(\lfloor \alpha|\CC_{n}(v)| \rfloor) \\
				& \geq \frac{\alpha}{|F|}\sum_{v \in V}|\CC_{n}(v)| - \frac{|V|}{|F|} \\
				& \geq \alpha\frac{|\bigcup_{v \in V}\CC_{n}(v)|}{|F|}- \frac{|F|}{|F||B(1_G,n)|} \\
				& \geq \alpha\frac{|F \setminus R|}{|F|}- \frac{1}{|B(1_G,n)|} \\
				& \geq \alpha(1-\frac{1}{2n}) - \frac{1}{2n} \\
				& \geq \alpha-\frac{1}{n}.\\
			\end{align*}
			
			While from the right side:
			\begin{align*}
				\dens(1,P)&   \leq \frac{1}{|F|}\sum_{v \in V}(\lfloor \alpha|\CC_{n}(v)| \rfloor +1) + \frac{|R|}{|F|} \\
				& \leq \frac{\alpha}{|F|}\sum_{v \in V}(|\CC_{n}(v)|)+ \frac{|V|}{|F|} + \frac{1}{2n} \\
				& \leq \alpha\frac{|\bigcup_{v \in V}\CC_{n}(v)|}{|F|}+ \frac{|F|}{|F||B(1_G,n)|} + \frac{1}{2n} \\
				& \leq \alpha + \frac{1}{n}. \\
			\end{align*}\end{proof}	
		Putting together Claims~\ref{claim.exists_config_condition},~\ref{claimbolas} and~\ref{claim.condition_implies_alpha}, we obtain that $X_\alpha\neq\emptyset$ which completes the proof of Theorem~\ref{theorem.densities}. \end{proof}
	
	\begin{remark*}
	In the case where $\alpha$ is a computable number, the subshift $X_{\alpha}$ given in the previous proof is $G$-effectively closed.
	\end{remark*}
	
	By noting that in the case of a group of subexponential growth, the sequence of balls always forms a F\o lner sequence~\cite[Theorem~6.11.2]{ceccherini-SilbersteinC09}, we obtain the following result.
	
	\begin{corollary}\label{corollary subexponential}
		Let $G$ be a group of subexponential growth, for every set of generators $S$ and $\alpha \in [0,1]$ there exists a non-empty $G$-subshift with uniform density.
	\end{corollary}

	\begin{definition}
		Let $G$ be a group and $S$ a finite set of generators. The rate of convergence of a subshift $X$ with uniform density $\alpha$ is the function $$\theta_X(n) := \inf\{ k \in \NN \mid \sup_{g \in G, x \in X} |\dens(1,B(g,k),x)-\alpha| \leq \frac{1}{n}  \}.$$
	\end{definition}
	
	As the construction given in Theorem~\ref{theorem.densities} is explicit, we can give bounds for the rate of convergence of $X_{\alpha}$. Indeed, let $\gamma$ denote the growth of a group $G$, that is, $\gamma(k) = |B(1_G,k)|$ for a fixed set of generators $S$. Let $x \in X$, $g \in G$ and $B_k := B(g,k)$. By definition of $X_{\alpha}$, if $2n|\partial_{B(1_G,5^n)}B_k| < |B_k|$ then  $|\dens(1,B_k,x)-\alpha|<\frac{1}{n}$ for each $x \in X$ and $g \in G$.
	
	 As $\partial_{B(1_G,5^n)}B_k = B_k \setminus B_{k-2\cdot5^n}$ if $k \geq 2\cdot 5^n$, we obtain that: $$\theta_{X_{\alpha}}(n) = \inf\{ k \geq 2\cdot 5^n \mid 2n(\gamma(k)-\gamma(k-2\cdot 5^n)) < \gamma(k) \}.$$
	 
	 Therefore a lower bound is always $\theta_{X_{\alpha}}(n) \geq 2\cdot 5^n$. The upper bound depends on the growth rate of the group. For instance, if $G$ has polynomial growth then $\theta_{X_{\alpha}}(n) = O(n\cdot 5^n).$ Indeed, if $\gamma(k) = k^d$ for some $d \geq 1$ we can write: 
	 
	 \begin{align*}
	 2n(k^d-(k-2\cdot 5^n)^d) < k^d & \iff 1-( 1- \frac{2\cdot 5^n}{k} )^d < \frac{1}{2n}\\ & \iff ( 1- \frac{2\cdot 5^n}{k} )^d > 1-\frac{1}{2n}
	 \end{align*}
	 
	 By Bernoulli's inequality, $( 1- \frac{2\cdot 5^n}{k} )^d \geq 1 - \frac{2d 5^n}{k}$. Hence it suffices to choose $k > 4nd5^n$. This shows that $\theta_{X_{\alpha}}(n) = O(n\cdot 5^n)$. In the case of a group of subexponential growth, an upper bound can be computed with an analogous reasoning given the exact rate of growth $2^{k^{\beta}}$ for some $\beta \in (0,1)$.	

	Sturmian sequences are classical examples of aperiodic sequences~\cite{fogg2002substitutions}. As $X_{\alpha}$ shares this uniform density property which makes it similar to Sturmian sequences, it is a natural question to ask if it satisfies a form of aperiodicity.
	
	\begin{proposition}
		\label{proposition.X_alpha_weakly_aperiodic}
		Let $\alpha \in [0,1]\setminus \mathbb{Q}$. Then $X_{\alpha}$ is weakly aperiodic.
	\end{proposition}
	
	\begin{proof}
		Suppose there exist a configuration $x \in X_{\alpha}$ and an integer $n \in \NN$ such that $|Orb_{\sigma}(x)| = n$. Let $D := \{g_i\}_{1 \leq i \leq n} \subset G$ such that each $\sigma_{g_i}(x)$ represents a different element of $Orb_{\sigma}(x)$, with $g_1 = 1_G$. Consider also $\alpha' = \dens(1,x|_D) \in \mathbb{Q}$ and $N = \max_{1 \leq i \leq n}{d(1_G,g_i)}$.

		Let $m \in \NN$ such that $\frac{2}{m} < |\alpha - \alpha'|$ and $ 5^m > \frac{N}{2}$. Recall that $K_m := B(1_G,5^m)$ and consider a finite subset $F \subset G$ such that $2m|\partial_{K_m}F|<|F|$ -- by amenability of $G$ such a subset always exists. As $x \in X_{\alpha}$ we get that $|\dens(1,x|_F)-\alpha|<\frac{1}{m}$. Let $V = Int(F,B(1_G,N)) \cap stab_{\sigma}(x)$ and $R = F \setminus VD$. Note that by definition of $N$, $VD = \bigcup_{v \in V}vD \subset F$ and that as each $v \in stab_{\sigma}(x)$ then $\dens(1,x|_{VD}) = \dens(1,x|_D) = \alpha'$. We obtain: $$\dens(1,x|_D)\frac{|VD|}{|F|} \leq \dens(1,x|_F) \leq \dens(1,x|_D)\frac{|VD|}{|F|}+\frac{|R|}{|F|}$$
		
		Let $g \in Int(F,K_m)$. Since the configuration $x$ is supposed to have finite orbit $\{x,\sigma_{g_2}(x),\dots,\sigma_{g_n}(x)\}$, there exists $l \in \{1,\dots,n\}$ such that $\sigma_{g^{-1}}(x) = \sigma_{g_{l}}(x)$. Therefore $g_l^{-1}g^{-1} \in stab_{\sigma}(x)$ which is a subgroup, thus $gg_{l} \in stab_{\sigma}(x)$. As $d(g,gg_{l}) \leq N$ and $gg_{l} \in V$ then we conclude that $Int(F,K_m) \subset VD$ (because we have chosen $g_1=1_G$) and therefore $R \subset \partial_{K_m}F$.
		
		Similarly to the previous proof, we bound each side using this relation, obtaining:
		
		\begin{alignat*}{2}
			\alpha'\frac{|VD|}{|F|}  &\leq \dens(1,x|_F) \leq &\alpha'\frac{|VD|}{|F|}+\frac{|R|}{|F|} \\
			\alpha'\frac{|F \setminus R|}{|F|}  &\leq \dens(1,x|_F) \leq & \alpha'+\frac{|\partial_{K_m}F|}{|F|} \\
			\alpha'(1-\frac{1}{2m})  &\leq \dens(1,x|_F) \leq & \alpha'+\frac{1}{2m} \\
			\alpha' -\frac{1}{2m}  &\leq \dens(1,x|_F) \leq & \alpha'+\frac{1}{2m} 
		\end{alignat*}
		
		Therefore $|\dens(1,x|_F)-\alpha'|<\frac{1}{m}$ and $|\dens(1,x|_F)-\alpha|<\frac{1}{m}$ which implies that $|\alpha - \alpha'| < \frac{2}{m}$ contradicting the definition of $m$.	 
	\end{proof}
	
	In the case of $\ZZ^2$, the subshift $X_\alpha$ defined in the proof of Theorem~\ref{theorem.densities} is not strongly aperiodic, since it contains the following configurations with non-trivial stabilizer: take a bi-infinite Sturmian word and repeat it vertically so that a configuration $x\in\{0,1\}^{\ZZ^2}$ is defined. Then $x\in X_\alpha$ since no forbidden pattern defining $X_\alpha$ appears in $x$. Thus Proposition~\ref{proposition.X_alpha_weakly_aperiodic} is in some sense the best we can do for this particular construction.
	
	The statement of Theorem~\ref{theorem.densities} itself requires amenability for the group $G$ in order to be meaningful, since we want the density to converge to $\alpha$ for every F\o lner sequence. Therefore, it doesn't say anything about non-amenable groups. For free groups, we can build configurations (and therefore, subshifts) with uniform density by constructing a regular covering tree and putting a Sturmian sequence on every level of this tree. Nevertheless, we still don't know if this kind of construction is always possible. To our knowledge the following question remains open:
	\begin{question*}
		Let $G$ be an infinite group generated by a finite set $S$ and $\alpha\in[0,1]$. Does there exist a subshift $Y_\alpha\subset \{0,1\}^{G}$ with uniform density $\alpha$?
	\end{question*}

	\textbf{Acknoledgements}: This work was partially supported by the ANR project CoCoGro (ANR-16-CE40-0005).
	
	\bibliographystyle{plain}

\end{document}